\documentclass[11pt]{article}
\usepackage[latin10]{inputenc}
\usepackage {graphicx}
\usepackage{a4}
\usepackage{amsmath, amssymb, amsthm}
\usepackage{dsfont}
\usepackage{subfig}

\theoremstyle{plain}

\newtheorem{lemma}{Lemma}[section]
\newtheorem{theorem}[lemma]{Theorem}
\newtheorem{proposition}[lemma]{Proposition}

\theoremstyle{definition}

\newtheorem*{definition}{Definition}
\newtheorem*{remark}{Remark}
\newtheorem*{example}{Example}

\usepackage{hyperref}
\usepackage{tikz}
\usetikzlibrary{positioning, arrows}

\renewcommand{\labelenumi}{(\roman{enumi})}

\begin{document}

\title{On a game on graphs}

\author{Felix G\"unther\footnote{Berlin Mathematical School, Institut f\"ur Mathematik, MA 2-2, Technische
Universit\"at Berlin, Stra{\ss}e des 17. Juni 136, 10623 Berlin, Germany.}$\;^{,1}$ \and Irina Musta\textcommabelow t\u a  \footnotemark[1]$\;^{,2}$}

\date{}
\maketitle

\footnotetext[1]{Supported by the Deutsche Telekom Stiftung. E-mail: fguenth@math.tu-berlin.de}
\footnotetext[2]{Supported by the Berlin Mathematical School. E-mail: mustata@math.tu-berlin.de}

\begin{abstract}
\noindent
We start with the well-known game below:
Two players hold a sheet of paper to their forehead on which a positive integer is written. The numbers are consecutive and each player can only see the number of the other one. In each time step, they either say nothing or tell what number they have. Both of them will eventually figure out their number after a certain amount of time. The game is rather cooperative than competitive, and employs the notions of \emph{common knowledge} and \emph{mutual knowledge}. We generalize this game to arbitrary (directed and non-directed) simple graphs and try to establish for which graphs one or both of them will figure out the solution, and how long they do need to find it. We give a complete answer for the case of two players, even if they are both allowed to discuss before the start of the game. \\ \vspace{0.5ex}

\vspace{0.5ex}

\end{abstract}

\renewcommand{\labelenumi}{(\roman{enumi})}

\raggedbottom
\setlength{\parindent}{0pt}
\setlength{\parskip}{1ex}


\section{Introduction}\label{sec:intro}


\subsection{The original game}\label{ssec:orig}
We consider a two-player game, also known as a Conway paradox, originally introduced by Conway et al. in \cite{CP77} and analysed by van Ende-Boas, Groenendijk and Stokhof in \cite{BGS80}. Two players, $A$ and $B$, get each a sheet of paper on their forehead containing one element of a pair of two consecutive positive integers. They can only see the number on the other player. Now each player tries to figure out which number is on their sheet. The only information they get is that in each time step (which is the same for both, e.g., given by a clock) they are allowed to either say nothing or to say the correct number they have. After some time both of them will eventually figure out their number: as a base case example, if one of them has $1$ on their forehead, the other player will know immediately that they have $2$, in the next step the other player will know that they have $1$. If no one has $1$, both of them will know that no one has $1$ after the first time step. In the second round a player will know 
immediately 
the right answer if the other has $2$ and so on. This is formalized in the Theorem on page 5 in \cite{BGS80}.

The paradox is that although they might figure out immediately that both numbers are positive without being told so, they cannot figure out the correct solution unless a lower (or upper) bound is given. In each step, they simulate playing a smaller pair of consecutive integers, hence the problem has an essentially finite descent approach, the direction being however bottom-up.

Equivalently, we can consider the set of positive integers as a semi-infinite path (a very simple graph structure). If an upper bound is given as well, the path is finite. Now both players are sitting on an edge knowing only the position of the other player and try to figure out their own position. The strategy we described above consists of cutting off leaves (or edges, as in \cite{BGS80}, but the term ``leaves'' is better suiting our purpose).


\subsection{Background}\label{ssec:theory}

One classical formulation of such a problem we find in the book of Fagin, Halpern, Moses and Vardi \cite{FHMV95}: A family has $n$ children playing outdoors one day. Some of them (for instance $k$) have gotten some mud on their foreheads, and whoever does, must clean up before dinner. Everyone can only see the faces of others, and no discussion of whose forehead is muddied happens. At some point the fathers goes out, looks at them, and says ``At least one of you has mud on their forehead.'' Nothing happens for the first $k-1$ time intervals, and during the $k$th, the muddied children go and wash their faces. The question is, what was the reasoning behind this. After all, the father does not impart new information to any child (assuming $k\geq 2$). However, what he says makes each child aware that the others have the same information.

As for example in \cite{FHMV95}, we can formalize the above, by stating that these paradoxes play on the notion of mutual vs. common knowledge. While mutual knowledge is limited to one epistemic level (``$A$ and $B$ know a fact $\varphi$''), common knowledge presumes an ad infinitum iteration of the statement: ``$A$ knows that $B$ knows that $A$ knows... the fact $\varphi$'' and the converse. Hence, if in the game from \ref{ssec:orig} no bounds for the pair of numbers are given, it is mutual knowledge if, say,  $A$ and $B$ are assigned the pair $(5,6)$ that the numbers are positive, but, since lacking a base case,  no common knowledge can be derived from it. Thus, it is not possiible that the players work their way up to a solution.

Introduced by Lewis in \cite{L69}, and also analysed in \cite{FHMV95}, common knowledge plays a central role in daily life, as it is necessary for the interaction of a group of agents and for establishing convention. For instance, one such convention is that the red colour of the street light means ``stop''. Here, it does not suffice that each driver and pedestrian is aware of this, but that each is aware the other is aware, and so on (that is, it is safe to cross the road as a pedestrian when it is red for vehicles, since any pedestrian assumes any driver is aware of the convention).

Several models exist for this (and epistemological aspects in general), one of them using an extension of modal logic \cite{FHMV95,MH95}, leading to a graph representation based on the concept of possible worlds (from the perspective of any agent, given their current level of information) \cite{FHMV95}. For a detailed description we direct the reader to the given literature.


\subsection{Organization of the paper}\label{ssec:struct}

Our purpose it to generalize the game from~\ref{ssec:orig} to finite graphs, where the players are given the endpoint labels of an arbitrary edge, with the aim to guess it. There are essentially two variants of this game: Either they are both allowed to speak in each time step, or they talk one after another, the starting player being known. Note that in the base problem above, there is not a big difference between both variants.

In Section~\ref{sec:undirected}, we will discuss the case of two players on finite (simple) graphs. First, we show  in Section~\ref{ssec:strategy} that at least one player will figure out the solution for any edge of the graph if and only if the graph is a forest. The result is the same if both players are allowed to discuss before the start of the game, already knowing the graph. Moreover, cutting off leaves is the most effective strategy in the sense of knowing the solution as quickly as possible. It will follow that cutting off leaves is also the strategy they come up with if they are not allowed to discuss before. If they speak in turns, the strategy is very similar, but slightly differs with respect whose turn is it. Note that if both players can discuss a strategy, it does not matter whether they are allowed to talk at the same time or not. We describe the strategy of cutting off leaves in more detail in the proof of Theorem~\ref{thm:tree_only}.

In Section~\ref{ssec:time} we show that for any tree there exists a strategy such that both players can determine their position. In the case they do not discuss before, we give a criterion which player will answer first, how long the player needs and if the other one also can find out their position. Here it makes a difference whether they talk simultaneously or alternately in each time step.

The ideas we developed in the case of undirected graphs work for directed graphs in a similar way. Only the concept of a path will slightly differ. We state the corresponding results in Section~\ref{sec:directed}.

We close our paper with remarks to some possible future directions in Section~\ref{sec:future}.

\section{Two-player game on simple graphs}\label{sec:undirected}

Let $G=(V,E)$ be a finite simple graph. The players $A$ and $B$ are placed on nodes $u\neq v \in V$  not knowing on which position in $V$ they are, but knowing that their assigned nodes are adjacent. For simplicity, we identify the players with the vertices they are placed on. Now both try to figure out on which position they are knowing only $G$ and the vertex of the other. Since the connected component of the graph they are placed on is common knowledge, we will assume that $G$ is connected.


\subsection{Strategy}\label{ssec:strategy}

In the following, we will simultaneously handle the cases whether they are taking turns in speaking or are making their statements at the same time. For the next lemma, it does not matter whether both players have discussed a strategy before or come up with one independently.

\begin{lemma}\label{lem:edge_guessed}
Suppose both players have a strategy such that they are knowing the right answer when they say so. Let $A$ guess their position after $n$ time steps where $B$ has not said anything before time $n$. Then, if $A$ had been placed on a different vertex adjacent to $B$, $B$ would have guessed their position correctly not later than time $n-1$.
\end{lemma}
\begin{proof}
Assume the contrary. Then $A$ would have the same information in both situations, namely the position of $B$ and that $B$ had not said anything in the $n-1$ steps before. Thus, $A$ would guess the same position for both cases, contradicting that they knows the correct answer.
\end{proof}

The following proposition will describe the graphs where such a game is possible to finish, regardless of the chosen positions:

\begin{theorem}\label{thm:tree_only}
All edges of a simple graph $G$ can be guessed correctly by at least one of the players if and only if $G$ is a tree.
\end{theorem}
\begin{proof}
$\Rightarrow$: Suppose the contrary. Since $G$ is connected, $G$ contains a cycle. We choose an edge of the cycle and a positioning of the two players on it in such a way, that the time $n$ needed for guessing the edge correctly is minimal. By Lemma~\ref{lem:edge_guessed}, at least one neighboring edge of the cycle is guessed in at most $n-1$ steps (assuming the right positioning of players), contradiction.

$\Leftarrow$: In the following, we consider a tree $G$. For any $X \subseteq V(G)$ we define $L(X)$ as the set of all leaves of $G$ contained in $X$. 

Suppose first that they talk simultaneously. If both players have not discussed a strategy before, then they can at least cut off all vertices of degree $1$ of the graph after both players have said nothing, since one of them would now the answer immediately if they know the other one to be on a leaf. The strategy corresponding to cutting off all vertices of (current) degree $1$ will be called \emph{cutting off leaves}. Thus, all edges of a tree can be guessed with this strategy.

Suppose now they speak alternately, $A$ being the starting player. Since $G$ is a tree and thus bipartite, we take the corresponding decomposition of $V$ into $V_A$ and $V_B$, such that $A \in V_A$. After $A$ said nothing in the first round, $B$ knows that they are not in $L(V_B)$ and can cut all these leaves off. Both players can then remove $G(L(V_B))$ and update $G$. If they still do not know where they are, $A$ can now cut off all vertices of $L(V_A)$, followed again by removing $G(L(V_A))$, and so on. Since $G$ is a tree, they cannot get stuck. We call this strategy \emph{cutting off leaves} as well.
\end{proof}

The following theorem shows that they cannot exclude more vertices in a step, such that cutting off leaves is exactly the strategy they will play if they cannot talk about a strategy before.

\begin{theorem}\label{thm:main_strategy}
Let $G=(V,E)$ be a tree and suppose both players have a correct strategy. Take any positioning of $A$ and $B$. Partition $V$ into two independent sets $V_A$ and $V_B$ such that $A\in V_A$. Now we direct any edge to the player who knows the answer first (if they both say the correct answer in the same moment, the edge will be bidirected) and label it by the time in which the player says the correct answer. Here the choice of players on the edge is given by the partition of $V$ into $V_A$ and $V_B$.

Then the labels are strictly increasing along directed paths. Moreover, unless $G$ consists of only one edge, there exist either one or two vertices with all incident edges going inward. For all other vertices, there is exactly one incident edge going outward. If there are two vertices with all incident edges going inward, then they are connected by a bidirectional edge. Also, this is the only case an edge with two directions can appear.

In particular, any strategy is a variant of cutting off leaves, and both players come up with cutting off leaves if they do not discuss before.
\end{theorem}
\begin{proof}
That labels are strictly increasing along directed paths follows directly from Lemma~\ref{lem:edge_guessed}. Of course, in the case they speak alternately, the label corresponding to an edge pointing to the player speaking at odd or even time can only be odd or even, respectively.

The endpoint $v$ of an unextendible directed path is defined as a vertex with all incident edges going inward. By Lemma~\ref{lem:edge_guessed}, any path from a point $v^\prime \neq v$ to $v$ is directed to $v$. If no incident edge has two directions, $v$ is the only vertex with this property. If there is such one, the other vertex $v^*$ incident to this edge also has the property of having all incident edges ingoing.

Lemma~\ref{lem:edge_guessed} shows that if one edge incident to a vertex is outgoing, all other edges incident to this vertex have to be ingoing and cannot be outgoing. In particular, bidirected edges can only connect vertices with all incident edges going inward. No other adjacent edge can have two directions. It follows that $v$ and $v^*$ are the only vertices with all incident edges being ingoing, and the edge incident to both is the only one with two directions.

We directly see that in time step $n$ only leaves of the graph consisting of edges with labels greater or equal than $n$ are guessed, and the player on the interior endpoint says the correct answer (unless the updated graph consists of only one edge by this point and we cannot speak about the interior). Thus, any strategy is a variant of cutting off leaves, and the latter one is the fastest one among all strategies. Since both players can at least cut off all leaves of the corresponding set $V_A$ or $V_B$ in their step when they do not discuss before, cutting off leaves is exactly the strategy they come up with if they have not discussed before.
\end{proof}

\begin{example}

Figure~\ref{fig:simultaneously} and Figure~\ref{fig:alternately} visualize the strategy of cutting of leaves when speaking simultaneously or alternately, respectively, in the notation of Theorem~\ref{thm:main_strategy}. White vertices correspond to $V_A$ and black to $V_B$. Player $A$ starts and talks at odd time, player $B$ at even time.

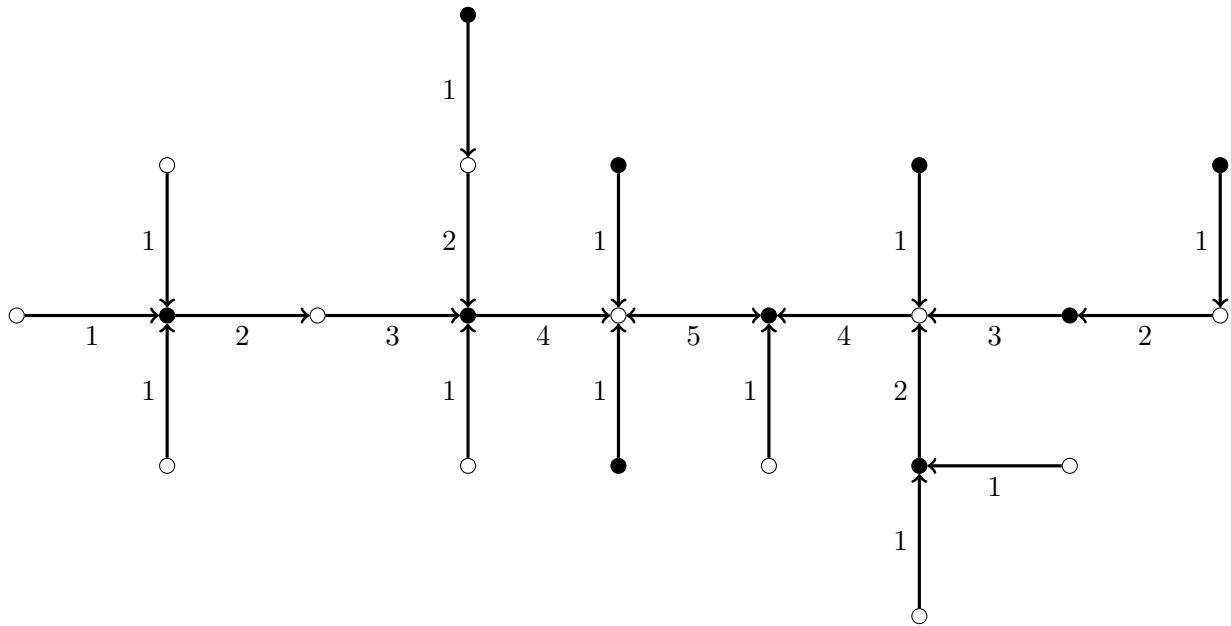
\begin{figure}[htbp]
\begin{center}
\beginpgfgraphicnamed{simultaneously}
\begin{tikzpicture}
[white/.style={circle,draw=black,fill=white,thin,inner sep=0pt,minimum size=2.0mm},
black/.style={circle,draw=black,fill=black,thin,inner sep=0pt,minimum size=2.0mm},scale=2]
\node[white] (w1)
at (2,-2) {};
\node[white] (w2)
 at (-3,-1) {};
\node[white] (w3)
 at (-1,-1) {};
\node[white] (w4)
 at (1,-1) {};
\node[white] (w5)
 at (3,-1) {};
\node[white] (w6)
 at (-4,0) {};
\node[white] (w7)
 at (-2,0) {};
\node[white] (w8)
 at (0,0) {};
\node[white] (w9)
 at (2,0) {};
\node[white] (w10)
 at (4,0) {};
\node[white] (w11)
 at (-3,1) {};
\node[white] (w12)
 at (-1,1) {};
\node[black] (b1)
 at (0,-1) {};
\node[black] (b2)
 at (2,-1) {};
\node[black] (b3)
 at (-3,0) {};
\node[black] (b4)
 at (-1,0) {};
\node[black] (b5)
 at (1,0) {};
\node[black] (b6)
 at (3,0) {};
\node[black] (b7)
 at (0,1) {};
\node[black] (b8)
 at (2,1) {};
\node[black] (b9)
 at (4,1) {};
\node[black] (b10)
 at (-1,2) {};
\path[->,very thick]
	(w5) edge node[midway,below] {$1$} (b2)
	(w6) edge node[midway,below] {$1$} (b3)
	(b3) edge node[midway,below] {$2$} (w7)
	(w7) edge node[midway,below] {$3$} (b4)
	(b4) edge node[midway,below] {$4$} (w8)
	(w10) edge node[midway,below] {$2$} (b6)
	(b6) edge node[midway,below] {$3$} (w9)
	(w9) edge node[midway,below] {$4$} (b5)
	(w2) edge node[midway,left] {$1$} (b3)
	(w11) edge node[midway,left] {$1$} (b3)
	(w3) edge node[midway,left] {$1$} (b4)
	(b10) edge node[midway,left] {$1$} (w12)
	(w12) edge node[midway,left] {$2$} (b4)
	(b1) edge node[midway,left] {$1$} (w8)
	(b7) edge node[midway,left] {$1$} (w8)
	(w4) edge node[midway,left] {$1$} (b5)
	(w1) edge node[midway,left] {$1$} (b2)
	(b2) edge node[midway,left] {$2$} (w9)
	(b8) edge node[midway,left] {$1$} (w9)
	(b9) edge node[midway,left] {$1$} (w10);
\path[<->,very thick]
	(w8) edge node[midway,below] {$5$} (b5);
\end{tikzpicture}
\endpgfgraphicnamed
\caption{Notation of Theorem~\ref{thm:main_strategy} for cutting off leaves when speaking simultaneously}
\label{fig:simultaneously}
\end{center}
\end{figure}

\begin{figure}[htbp]
\begin{center}
\beginpgfgraphicnamed{simultaneously}
\begin{tikzpicture}
[white/.style={circle,draw=black,fill=white,thin,inner sep=0pt,minimum size=2.0mm},
black/.style={circle,draw=black,fill=black,thin,inner sep=0pt,minimum size=2.0mm},scale=2]
\node[white] (w1)
at (2,-2) {};
\node[white] (w2)
 at (-3,-1) {};
\node[white] (w3)
 at (-1,-1) {};
\node[white] (w4)
 at (1,-1) {};
\node[white] (w5)
 at (3,-1) {};
\node[white] (w6)
 at (-4,0) {};
\node[white] (w7)
 at (-2,0) {};
\node[white] (w8)
 at (0,0) {};
\node[white] (w9)
 at (2,0) {};
\node[white] (w10)
 at (4,0) {};
\node[white] (w11)
 at (-3,1) {};
\node[white] (w12)
 at (-1,1) {};
\node[black] (b1)
 at (0,-1) {};
\node[black] (b2)
 at (2,-1) {};
\node[black] (b3)
 at (-3,0) {};
\node[black] (b4)
 at (-1,0) {};
\node[black] (b5)
 at (1,0) {};
\node[black] (b6)
 at (3,0) {};
\node[black] (b7)
 at (0,1) {};
\node[black] (b8)
 at (2,1) {};
\node[black] (b9)
 at (4,1) {};
\node[black] (b10)
 at (-1,2) {};
\path[->,very thick]
	(w5) edge node[midway,below] {$2$} (b2)
	(w6) edge node[midway,below] {$2$} (b3)
	(b3) edge node[midway,below] {$3$} (w7)
	(w7) edge node[midway,below] {$4$} (b4)
	(b4) edge node[midway,below] {$5$} (w8)
	(b5) edge node[midway,below] {$5$} (w8)
	(w10) edge node[midway,below] {$2$} (b6)
	(b6) edge node[midway,below] {$3$} (w9)
	(w9) edge node[midway,below] {$4$} (b5)
	(w2) edge node[midway,left] {$2$} (b3)
	(w11) edge node[midway,left] {$2$} (b3)
	(w3) edge node[midway,left] {$2$} (b4)
	(b10) edge node[midway,left] {$1$} (w12)
	(w12) edge node[midway,left] {$2$} (b4)
	(b1) edge node[midway,left] {$1$} (w8)
	(b7) edge node[midway,left] {$1$} (w8)
	(w4) edge node[midway,left] {$2$} (b5)
	(w1) edge node[midway,left] {$2$} (b2)
	(b2) edge node[midway,left] {$3$} (w9)
	(b8) edge node[midway,left] {$1$} (w9)
	(b9) edge node[midway,left] {$1$} (w10);
\end{tikzpicture}
\endpgfgraphicnamed
\caption{Notation of Theorem~\ref{thm:main_strategy} for cutting off leaves when speaking alternately}
\label{fig:alternately}
\end{center}
\end{figure}
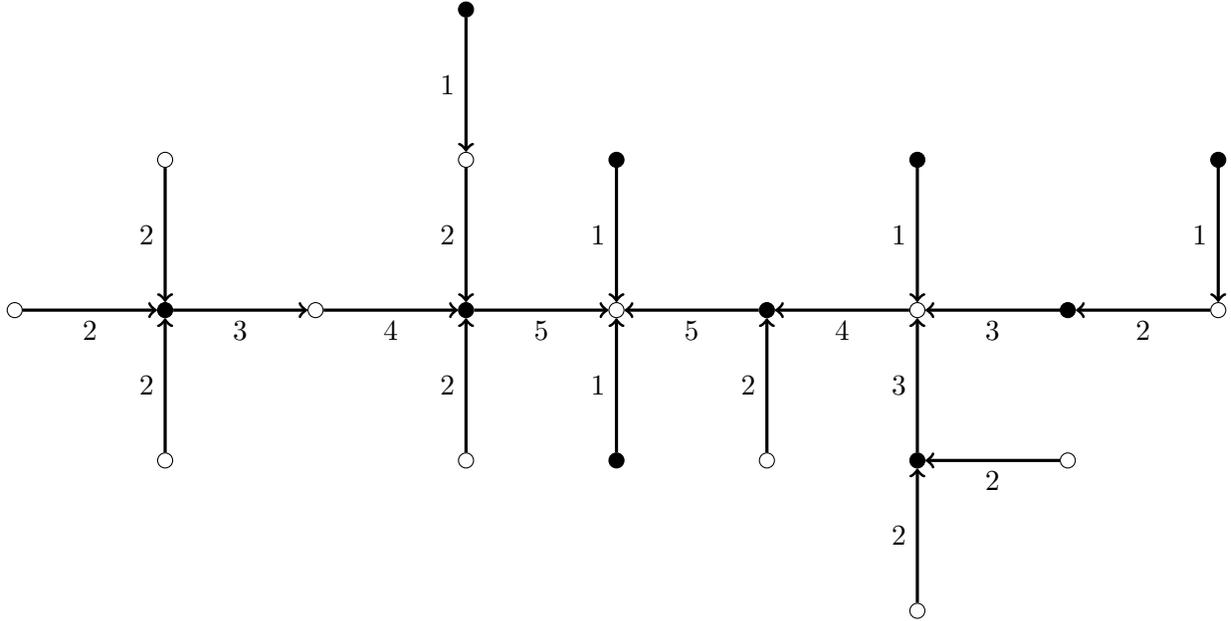

\end{example}
We can summarize the difference between the cases where the players speak simultaneously or alternatively, as follows:
\begin{itemize}
\item If the game is played simultaneously: At each time step, $G$ is updated by removing all leaves of $G$.
\item If the game is played alternately: After each odd time step, $G$ is updated by removing all leaves of $G$ that are in $V_B$, whereas after each even one, all vertices of $V_A$ that are leaves of $G$ get removed.
  \end{itemize}
\begin{remark}
Suppose $G$ is a directed tree equipped with a labeling fulfilling the statement of Theorem~\ref{thm:main_strategy}, respecting the parity in the case the players speak alternately, meaning that the starting player can only guess their position at odd and the other one at even times. Then this labeling corresponds to a strategy: Player $A$ ($B$) looks at all edges incident to the position of $B$ ($A$), and if an edge with label $n$ is pointing away, they indicate at time $n$ the endpoint of this edge unless $B$ ($A$) has said their position before (we do not specify yet what they will say after the other player has told their position, assuming for the moment the game stops at this point). Note that  there is at most one such edge, since at most one edge incident to a vertex is outgoing. Since each edge is directed, at least the player on the vertex where the edge is ingoing will say the answer after some time. The strategy is correct, since by construction, all ingoing edges to the position of $B$ ($A$) are 
guessed by $B$ ($A$) themselves before.
\end{remark}

Finally, note that if $A$ and $B$ can discuss before, the cases when speaking simultaneously or alternately are essentially equivalent.

\begin{proposition}\label{prop:independent}
If both players can discuss a strategy before the start of the game, the set of edges of the graph they can find out correctly is the same independent of whether they speak simultaneously or alternately in each time step.
\end{proposition}
\begin{proof}
Suppose the players have a strategy if they speak simultaneously. They can adapt it to the alternative case as follows: they agree before that player $A$ speaks in each odd and player $B$ in each even time step. Moreover, $B$ ignores the information they just got from the immediately preceding round of $A$ just until after their own turn is over.

Now suppose they have a strategy when they speak alternately. Then in each time step, when one player would remain silent, the same says nothing when they are allowed to talk at the same time. This yields an equivalent strategy for the case of simultaneously speaking.
\end{proof}


\subsection{Who is first?}\label{ssec:time}

\begin{proposition}\label{prop:strategy_both}
For any tree $G=(V,E)$, there is a strategy allowing both players to determine their position.
\end{proposition}
\begin{proof}
Direct and label $G$ according to Theorem~\ref{thm:main_strategy} and the cutting off leaves strategy. Now multiply all labels by $|E|$. Establish a bijection $\phi$ between $E$ and the set $\left\{0,1,\ldots,|E|-1\right\}$. Now for each $e\in E$ add $ \phi(e)$ to the label of $e$. In the case they speak alternately, multiply the resulting number by $2$. Additionally, subtract $1$ if and only if the edge is pointing to the starting player. We obtain a labeling of $G$ fulfilling the statement of Theorem~\ref{thm:main_strategy} with all labels being different. By the remark after Theorem~\ref{thm:main_strategy}, the labeling corresponds to a strategy. Since all labels are different, the player who does not say the answer first can figure out their position correctly by the time the other one needed. Thus, both players can find out their position.
\end{proof}

Because both players can figure out their position with the right strategy by Proposition~\ref{prop:strategy_both}, we now consider the case they do not talk about a strategy before. We will discuss whether both players can get the right answer and how long the first player needs. This also gives an lower bound for the case of general strategies by Theorem~\ref{thm:main_strategy}.

\begin{theorem}\label{th:who_is_first}
Let $G=(V,E)$ be a tree and assume that both players play without discussing about a strategy before. Let $h_B(A)$ be the height of $A$ in the tree rooted in $B$ and $h_A(B)$ the height of $B$ in the tree rooted in $A$. Let $h^\prime_B(A):=h_B(A)+1$ if $h_B(A)$ is odd, and $h^\prime_B(A):=h_B(A)$ otherwise. In the same way, let $h^\prime_A(B):=h_A(B)+1$ if $h_A(B)$ is even, and $h^\prime_A(B):=h_A(B)$ otherwise.
\begin{enumerate}
\item Suppose they speak simultaneously in each step. If $h_B(A)>h_A(B)$, player $A$ will first know their position at time $h_A(B)$; if $h_B(A)=h_A(B)$, both of them will figure out their position at the same time $h_B(A)=h_A(B)$; if $h_B(A)<h_A(B)$, player $B$ says the answer first at time $h_B(A)$.

In the case that $h_B(A)>h_A(B)$, $B$ will figure out their position at time $h_A(B)+1$ as well if and only if there is no node other than $B$ in the tree rooted in $A$ with height $h_A(B)$. The analogous statement is true for the case $h_B(A)<h_A(B)$.

\item Suppose they speak alternately in each step, $A$ being the starting player. If $h^\prime_B(A)>h^\prime_A(B)$, player $A$ will first know their position at time $h^\prime_A(B)$; if $h^\prime_B(A)<h^\prime_A(B)$, player $B$ says the answer first at time $h^\prime_B(A)$.

In the case that $h^\prime_B(A)>h^\prime_A(B)$, $B$ will figure out their position at time $h^\prime_A(B)+1$ as well if and only if there is no node other than $B$ in the tree rooted in $A$ with height $h^\prime_A(B)$ or $h^\prime_A(B)-1$. The analogous statement is true for the case $h^\prime_B(A)<h^\prime_A(B)$.
\end{enumerate}
\end{theorem}
\begin{proof}
(i) Consider the longest simple path starting in $A$ going through $B$ and the longest simple path starting in $B$ going through $A$. By Theorem~\ref{thm:main_strategy}, the shorter of them will determine the time until one player knows the answer. If both have the same length, both players will know the answer at the same time (the corresponding edge has two directions), otherwise the edge is pointing to the player being on the boundary of the shortest path. The lengths of these paths are given by $h_A(B)$ and $h_B(A)$, respectively.

For the second part of the statement, assume without loss of generality that $A$ knows the answer first at time $n$. Then $B$ can figure out their position as well if and only if the edge incident to $A$ and $B$ is the only edge incident to $A$ with label $h_A(B)$ (note that any edge going out from $A$ has a label greater than $h_A(B)$ by Lemma~\ref{lem:edge_guessed}). This is the case if and only if there is no node other than $B$ in the tree rooted in $A$ with height $h_A(B)$.

(ii) In the same way as in (i), the label on the edge connecting $A$ and $B$ is given by the smaller of $h^\prime_B(A)$ and $h^\prime_A(B)$. The possible difference of $1$ between $h^\prime$ and $h$ is due to the fact that $A$ can say the answer at odd and $B$ at even times only. The edge is directed to $A$ if $h^\prime_A(B)$ is smaller, otherwise it is directed to $B$.

Without loss of generality, assume $A$ knows the answer first. As above, $B$ can figure out their position as well if and only if the edge incident to $A$ and $B$ is the only edge incident to $A$ with label $h^\prime_A(B)$. Remembering that $A$ can answer at odd times only, this is only the case if and only if there is no node other than $B$ in the tree rooted in $A$ with height $h^\prime_A(B)$ or $h^\prime_A(B)-1$. In the first case, the leaf determining the height of $B$ is in $V_B$, in the second case in $V_A$. As before, $V$ is partitioned into $V_A$ and $V_B$ such that each edge is incident to points in both sets and $A \in V_A$.
\end{proof}

\section{Two-player game on directed graphs}\label{sec:directed}

Let $G=(V,E)$ be a finite directed graph. The players $A$ and $B$ are placed on an edge in $E$. For simplicity, we identify the players with the vertices they are placed on. Now both try to figure out on which position they are knowing only $G$, the vertex of the other and the orientation of the edge.

In the following, we are only considering simple directed graphs if we speak about a directed graph. Note that the  graphs in Section~\ref{sec:undirected} occur as a special case, when any two adjacent vertices are connected by two edges with different orientation. Since the ideas are very similar, we will often refer to Section~\ref{sec:undirected} for proofs of analogous statements.

\begin{definition}
An edge together with a placement of $A$ and $B$ on the endpoints is called \textit{admissible}, if the orientation of the edge agrees with the orientation the players got assigned before.
\end{definition}

If the placement is clear (for example, when the position of one player is given or for vertex sets associated to $A$ or $B$), we will not give the details.

The following analogue of Lemma~\ref{lem:edge_guessed} is shown in the same way.

\begin{lemma}\label{lem:edge_guessed_directed}
Suppose both players have a strategy such that they are knowing the right answer when they say so. Assume $A$ guesses their position after $n$ time steps where $B$ has not said anything before time $n$. Then if $A$ had been placed on a vertex $A^\prime \neq A$ adjacent to $B$, such that the edge $(A^\prime,B)$ is admissible, $B$ would have guessed their position correctly no later than time $n-1$.
\end{lemma}

To transfer the theorems of the previous section to the case of directed graphs, we need a new concept of paths, see for example Figure~\ref{fig:zigzag}.

\begin{definition}
A \textit{zig-zag-path} is a path, where at each interior vertex either both incident edges are ingoing or both are outgoing. If the path is closed, we call it a \textit{zig-zag-cycle}.
\end{definition}

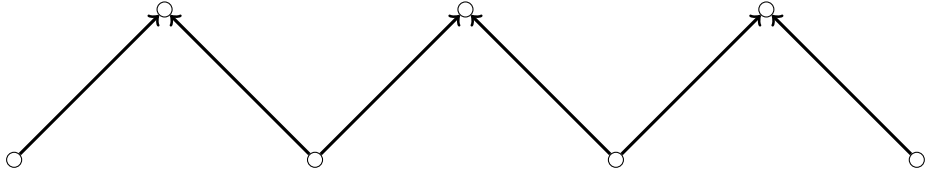
\begin{figure}[htbp]
\begin{center}
\beginpgfgraphicnamed{zigzag}
\begin{tikzpicture}
[white/.style={circle,draw=black,fill=white,thin,inner sep=0pt,minimum size=2.0mm}]
\node[white] (w1)
at (-6,-1) {};
\node[white] (w2)
 at (-4,1) {};
\node[white] (w3)
 at (-2,-1) {};
\node[white] (w4)
 at (0,1) {};
\node[white] (w5)
 at (2,-1) {};
\node[white] (w6)
 at (4,1) {};
\node[white] (w7)
 at (6,-1) {};
\path[->,very thick]
(w1) edge  (w2)
(w3) edge  (w2)
(w3) edge  (w4)
(w5) edge  (w4)
(w5) edge  (w6)
(w7) edge  (w6);
\end{tikzpicture}
\endpgfgraphicnamed
\caption{A zig-zag-path}
\label{fig:zigzag}
\end{center}
\end{figure}

We introduce now vertex sets $V_A$ and $V_B$.

\begin{definition}
$A \in V_A$, whereas any other vertex $A^\prime$ is in $V_A$ if and only if there is a zig-zag-path from $A$ to $A^\prime$ of even length such that the first edge (incident to $A$) of this path is admissible. $V_B$ is defined in the same way as $V_A$ by replacing $A$ by $B$. Equivalently, any point of $V_B$ can be reached by an appropriate zig-zag-path of odd length starting in $A$.
\end{definition}

\begin{remark}
Note that also for connected $G$ without zig-zag-cycles, there might be vertices being neither in $V_A$ nor in $V_B$. Moreover, the intersection $V_A \cap V_B$ might be non-empty. 
\end{remark} 

If both players have not discussed a strategy before, they can restrict to the subgraph  $G(V_A \cup V_B)$ immediately. Namely, this set can be constructed in the following way: Starting in $A$, add all adjacent vertices $B^\prime$ with an admissible orientation  of the edge $AB^\prime$ (i.e. all $B^\prime$ candidates for $B$). For each of them, add all adjacent vertices that are candidates for $A$. Continue this procedure. This set is known for both $A$ and $B$ as well knowing the position of each other.

This resulting graph $G'$ will be further modified as follows: Note that the only vertices with both indegree and outdegree non-zero are those belonging to both $V_A$ and $V_B$. Let $v\in V(G')$ with $\textnormal{indeg}(v)\neq 0$ and $\textnormal{outdeg}(v)\neq 0$. This vertex will now be removed and replaced (split) by $v_{\textnormal{in}}$ and $v_{\textnormal{out}}$, vertices incident exactly to the ingoing and outgoing edges of $v$, respectively. Since the procedure strictly decreases the number of vertices with mixed degree, it must be finite.

\begin{figure}[htbp]
   \centering
    \subfloat[before splitting]{
    \beginpgfgraphicnamed{before}
			\begin{tikzpicture}
			[white/.style={circle,draw=black,fill=white,thin,inner sep=0pt,minimum size=2.0mm}]
				\node[white] (w1) [label=left:$v$]
				at (0,0) {};
				\node[white] (w2)
				at (1,2) {};
				\node[white] (w3)
				at (1,-2) {};
				\node[white] (w4)
				at (-1,2) {};
				\node[white] (w5)
				at (2,0) {};
				\node[white] (w6)
				at (-1,-2) {};
				\path[->,very thick]
				(w1) edge  (w2)
				(w1) edge  (w3)
				(w1) edge  (w4)
				(w5) edge  (w1)
				(w6) edge  (w1);
			\end{tikzpicture}
		\endpgfgraphicnamed}
		\qquad
		\subfloat[after splitting]{
		\beginpgfgraphicnamed{after}
			\begin{tikzpicture}
			[white/.style={circle,draw=black,fill=white,thin,inner sep=0pt,minimum size=2.0mm}]
				\node[white] (w1) [label=left:$v_{\textnormal{out}}$]
				at (0,0) {};
				\node[white] (w11) [label=left:$v_{\textnormal{in}}$]
				at (4,0) {};
				\node[white] (w2)
				at (1,2) {};
				\node[white] (w3)
				at (1,-2) {};
				\node[white] (w4)
				at (-1,2) {};
				\node[white] (w5)
				at (6,0) {};
				\node[white] (w6)
				at (3,-2) {};
				\path[->,very thick]
				(w1) edge  (w2)
				(w1) edge  (w3)
				(w1) edge  (w4)
				(w5) edge  (w11)
				(w6) edge  (w11);
			\end{tikzpicture}
		\endpgfgraphicnamed}
   \caption[]{Split vertex $v$}
  \label{fig:splitv}
\end{figure}
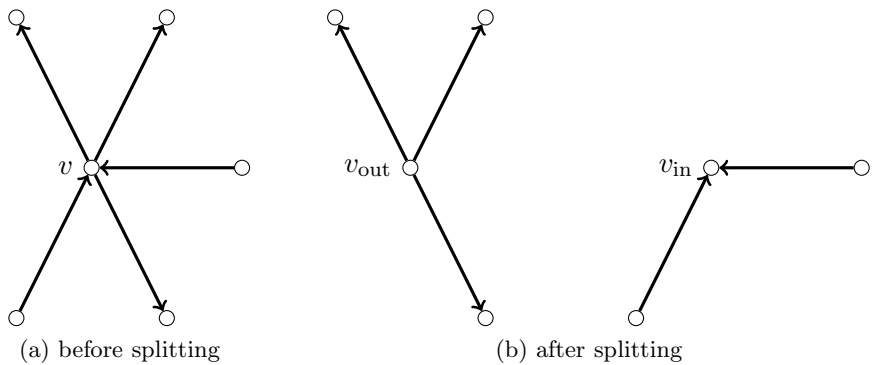

\begin{remark}
The graph $G_s$ obtained after applying the sequence of all possible splits to $G^\prime$ has only zig-zag paths and all edges are admissible.

It is important to notice that $G_s$ contains a zig-zag cycle if and only if $G^\prime$ does. Indeed, suppose $G_s$ has such a cycle. Then, there are two cases:
\begin{itemize}
\item No pair of vertices on the cycle stem from the same split vertex. Then, it must also have existed before splitting, since the procedure does not add new edges.
\item There exists at least a pair $v_{\textnormal{in}},v_{\textnormal{out}}$ that stem from the same vertex $v$ before the splitting, assuming without loss of generality that $v$ was the last such split vertex. In this case the cycle already existed before splitting, passing twice through $v$.
\end{itemize}
\end{remark}

\begin{figure}[htbp]
   \centering
    \subfloat[$G^\prime$]{
    \beginpgfgraphicnamed{G_before}
			\begin{tikzpicture}
			[white/.style={circle,draw=black,fill=white,thin,inner sep=0pt,minimum size=2.0mm},
			black/.style={circle,draw=black,fill=black,thin,inner sep=0pt,minimum size=2.0mm},
			gray/.style={circle,draw=black,fill=gray,thin,inner sep=0pt,minimum size=2.0mm}]
				\node[white] (w1) [label=left:$A$]
				at (-1,0) {};
				\node[black] (b1) [label=right:$B$]
				at (1,0) {};
				\node[black] (b2) [label=right:$t$]
				at (2,-2) {};
				\node[white] (w2) [label=left:$s$]
				at (-2,-2) {};
				\node[gray] (g1) [label=left:$u$]
				at (0,2) {};
				\node[gray] (g2) [label=below:$v$]
				at (0,-2) {};
				\path[->,very thick]
				(w1) edge  (b1)
				(g1) edge  (b1)
				(g2) edge  (b1)
				(w1) edge  (g1)
				(w1) edge  (g2)
				(w2) edge  (g2)
				(g2) edge  (b2);
			\end{tikzpicture}
		\endpgfgraphicnamed}
		\qquad
		\subfloat[$G_s$]{
		\beginpgfgraphicnamed{G_after}
			\begin{tikzpicture}
			[white/.style={circle,draw=black,fill=white,thin,inner sep=0pt,minimum size=2.0mm},
			black/.style={circle,draw=black,fill=black,thin,inner sep=0pt,minimum size=2.0mm}]
				\node[white] (w1) [label=left:$A$]
				at (-1,0) {};
				\node[black] (b1) [label=right:$B$]
				at (1,0) {};
				\node[black] (b2) [label=right:$t$]
				at (2,-2) {};
				\node[white] (w2) [label=left:$s$]
				at (-2,-2) {};
				\node[white] (g1w) [label=right:$u_{\textnormal{out}}$]
				at (1,2) {};
				\node[black] (g1b) [label=left:$u_{\textnormal{in}}$]
				at (-1,2) {};
				\node[white] (g2w) [label=below:$v_{\textnormal{out}}$]
				at (1,-2) {};
				\node[black] (g2b) [label=below:$v_{\textnormal{in}}$]
				at (-1,-2) {};
				\path[->,very thick]
				(w1) edge  (b1)
				(g1w) edge  (b1)
				(g2w) edge  (b1)
				(w1) edge  (g1b)
				(w1) edge  (g2b)
				(w2) edge  (g2b)
				(g2w) edge  (b2);
			\end{tikzpicture}
		\endpgfgraphicnamed}
   	\caption[]{Obtaining the splitgraph}
  	\label{fig:splitg}
\end{figure}
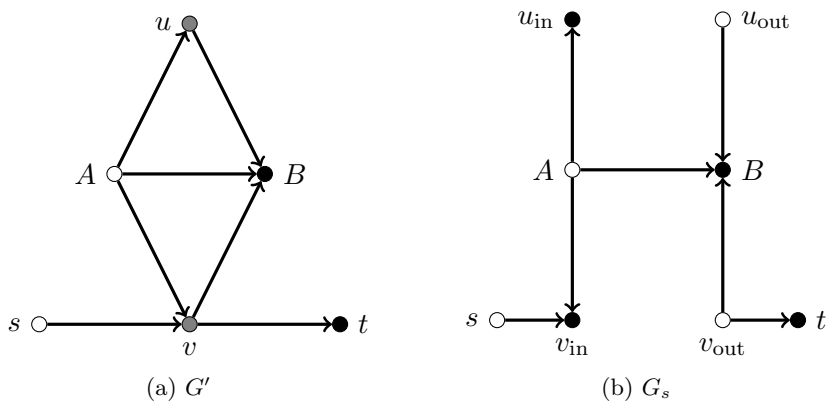

We can now easily reduce the problem to the undirected case and prove, with the exact same argument, the following theorem.
\begin{theorem}
A directed graph $G$ has all edges guessable under the above conditions if and only if it does not contain zig-zag cycles (i.e. it is a zig-zag forest). All conclusions relating to the number of necessary steps or to who makes the first guess hold. The comparison of the cases of having discussed a strategy before or not concludes along the same lines.
\end{theorem}


\section{Further directions}\label{sec:future}
One variation, for the two-player case, would be introducing ``cycle'' as a third allowed answer in addition to saying nothing or the results. This seems to allow further edges to be correctly guessed in a general graph. 

On the other hand, one can naturally generalize the problem to a multiplayer game where $n$ players are the vertices of a subgraph $H$ of $G$ with known isomorphism class. The players may or may not know their positions in $H$, otherwise the rules of the game stay the same. It would be interesting to describe how the game runs in at least a number of particular such cases.


\section*{Acknowledgment}
\addcontentsline{toc}{section}{Acknowledgment}

We would like to thank Thomas Hixon for numerous fruitful discussions.


\bibliographystyle{plain}
\bibliography{Game_on_graphs}

\end{document}